\DeclareSymbolFont{AMSb}{U}{msb}{m}{n}
\documentclass[11pt,a4paper,leqno,noamsfonts]{amsart}

   \makeatletter
   \renewcommand\@biblabel[1]{#1.}
   \makeatother
   \usepackage{multirow}

\usepackage[english]{babel}
\usepackage[dvipsnames]{xcolor}
\usepackage{graphicx,pifont,soul,physics} % physics to get bold + italics (\vb*{v})
\usepackage[utopia]{mathdesign}
\usepackage[a4paper,top=3cm,bottom=3cm,
           left=3cm,right=3cm,marginparwidth=60pt]{geometry}

\usepackage{fancyhdr}
\usepackage{float}

\usepackage[utf8]{inputenc}
\usepackage{braket,caption,comment,mathtools,stmaryrd}
\usepackage[usestackEOL]{stackengine}
\usepackage{multirow,booktabs,microtype,relsize}
\usepackage[colorlinks,bookmarks]{hyperref} 
      \hypersetup{colorlinks,
           citecolor=britishracinggreen,
            filecolor=black,
            linkcolor=cobalt,
            urlcolor=cornellred}
      \numberwithin{equation}{section}
\usepackage[capitalise]{cleveref}
\usepackage[colorinlistoftodos]{todonotes}%textwidth=35mm
\definecolor{antiquewhite}{rgb}{0.98, 0.92, 0.84}
\definecolor{buff}{rgb}{0.94, 0.86, 0.51}
\definecolor{palecopper}{rgb}{0.85, 0.54, 0.4}
\definecolor{fluorescentyellow}{rgb}{0.8, 1.0, 0.0}
\definecolor{bole}{rgb}{0.47, 0.27, 0.23}
\usepackage{amsmath,float}
\usetikzlibrary{decorations.markings}

\usepackage[vcentermath]{youngtab}
\input xy
\xyoption{all}

\usepackage{pgfkeys}
\usepackage{pgfopts}
\usepackage{ytableau}

\definecolor{cornellred}{rgb}{0.7, 0.11, 0.11}
\definecolor{britishracinggreen}{rgb}{0.0, 0.26, 0.15}
\definecolor{cobalt}{rgb}{0.0, 0.28, 0.67}
\DeclareSymbolFont{usualmathcal}{OMS}{cmsy}{m}{n}
\DeclareSymbolFontAlphabet{\mathcal}{usualmathcal}

\newcommand{\BA}{{\mathbb{A}}}

\newcommand{\BC}{{\mathbb{C}}}

\newcommand{\BR}{{\mathbb{R}}}

\newcommand{\conv}{\mathrm{conv}}

\DeclareMathOperator{\Hilb}{Hilb}

\DeclareMathOperator{\colength}{colength}

\DeclareFontFamily{OT1}{rsfs}{}
\DeclareFontShape{OT1}{rsfs}{n}{it}{<-> rsfs10}{}
\DeclareMathAlphabet{\curly}{OT1}{rsfs}{n}{it}
\renewcommand\hom{\mathscr{H}\kern-0.3em\mathit{om}}

\newcommand\Hom{\operatorname{Hom}}

\DeclareMathOperator{\lHom}{\mathscr{H}\kern-0.3em\mathit{om}}
\DeclareMathOperator{\RRlHom}{\mathbf{R}\kern-0.025em\mathscr{H}\kern-0.3em\mathit{om}}

\DeclareMathOperator{\lExt}{{\mathscr{E}\kern-0.2em\mathit{xt}}}



\usepackage[all]{xy}
\usepackage{tikz}
\usepackage{tikz-cd}
\usepackage{adjustbox}
\usepackage{rotating}
\usepackage{comment}

\usetikzlibrary{matrix,shapes,intersections,arrows,decorations.pathmorphing}
\tikzset{commutative diagrams/arrow style=math font}
\tikzset{commutative diagrams/.cd,
mysymbol/.style={start anchor=center,end anchor=center,draw=none}}

\tikzset{
shift up/.style={
to path={([yshift=#1]\tikztostart.east) -- ([yshift=#1]\tikztotarget.west) \tikztonodes}
}
}

\theoremstyle{definition}

\newtheorem*{lemma*}{Lemma}
\newtheorem*{theorem*}{Theorem}
\newtheorem*{example*}{Example}
\newtheorem*{fact*}{Fact}
\newtheorem*{notation*}{Notation}
\newtheorem*{definition*}{Definition}
\newtheorem*{prop*}{Proposition}
\newtheorem*{remark*}{Remark}
\newtheorem*{corollary*}{Corollary}

\newtheorem*{conventions*}{Conventions}

\newtheorem{definition}{Definition}[section]

\newtheorem{example}[definition]{Example}

\newtheorem{question}[definition]{Question}
\newtheorem{notation}[definition]{Notation}
\newtheorem{remark}[definition]{Remark}

\newtheoremstyle{thm} % <name> % (ambienti con dimostrazione)
        {3mm}% <Space above>
        {3mm}% <Space below>
        {\slshape}% <Body font> % 
        {0mm}% <Indent amount>
        {\bfseries}% <Theorem head font>
        {.}% <Punctuation after theorem head>
        {1mm}% <Space after theorem head>
        {}% <Theorem head spec (can be left empty, meaning 'normal')> 
\theoremstyle{thm}

\newtheorem{corollary}[definition]{Corollary}
\newtheorem{lemma}[definition]{Lemma}

\newtheorem{conj}{Conjecture}

\newtheoremstyle{ex} % <name> % (ambienti con dimostrazione)
        {3mm}% <Space above>
        {3mm}% <Space below>
        {}% <Body font> % \slshape
        {0mm}% <Indent amount>
        {\scshape}% <Theorem head font>
        {.}% <Punctuation after theorem head>
        {1mm}% <Space after theorem head>
        {}% <Theorem head spec (can be left empty, meaning 'normal')> 
\theoremstyle{ex}

\newtheoremstyle{sol} % <name> % (ambienti con dimostrazione)
        {3mm}% <Space above>
        {3mm}% <Space below>
        {}% <Body font> % 
        {0mm}% <Indent amount>
        {\scshape}% <Theorem head font>
        {.}% <Punctuation after theorem head>
        {1mm}% <Space after theorem head>
        {}% <Theorem head spec (can be left empty, meaning 'normal')> 
\theoremstyle{sol}

\usepackage{tikz}
\usepackage{xparse}

\newcount\tableauRow
\newcount\tableauCol

\newenvironment{Tableau}[1]{%
  \tikzpicture[scale=0.5,draw/.append style={thick,black},
                      baseline=(current bounding box.center)]
    \tableauRow=-1.5
    \foreach \Row in {#1} {
       \tableauCol=0.5
       \foreach\k in \Row {
         \draw[thin](\the\tableauCol,\the\tableauRow)rectangle++(1,1);
         \draw[thin](\the\tableauCol,\the\tableauRow)+(0.5,0.5)node{$\k$};
         \global\advance\tableauCol by 1
       }
       \global\advance\tableauRow by -1
    }
}{\endtikzpicture}

\newtheorem*{Acknowledgments*}{Acknowledgments}
\DeclareMathAlphabet\BCal{OMS}{cmsy}{b}{n}

\usepackage{amsmath,float}
\usetikzlibrary{decorations.markings}
\usepackage{verbatim}
\usepackage{graphicx}
\usepackage{subcaption}

\makeatletter
\newtheorem*{rep@theorem}{\rep@title}
\newcommand{\newreptheorem}[2]{%
\newenvironment{rep#1}[1]{%
 \def\rep@title{#2 \ref{##1}}%
 \begin{rep@theorem}}%
 {\end{rep@theorem}}}
\makeatother

\address{Centre for Mathematical Sciences, University of Cambridge, Wilberforce Road, CB3 0WA, Cambridge, United Kingdom}

\email{fr414@cam.ac.uk}

\title[Conjectural criteria for  the most singular points of the Hilbert schemes of points
]{Conjectural criteria for  the most singular points of the Hilbert schemes of points
}
\author{Fatemeh Rezaee}

\date{}

\begin{document}
\maketitle

\begin{abstract}
 We provide conjectural necessary and (separately) sufficient conditions for the Hilbert scheme of points of a given length to have the maximum dimension tangent space at a point. The sufficient condition is claimed for 3D 
 
and reduces the original problem to a problem in convex geometry. Proving either of the two conjectural statements will in particular resolve a long-standing conjecture by Briançon and Iarrobino back in the '70s for the case of the powers of the maximal ideal. Furthermore, for specific classes of lengths, we conjecturally classify points satisfying the conjectural sufficient conditions. This in particular (conjecturally) provides many new explicit families of examples of maximum dimension tangent space at a point of the Hilbert schemes of points of lengths strictly between two consecutive tetrahedral numbers ${3+k \choose 3}$.

\end{abstract}

{\hypersetup{linkcolor=black}
\tableofcontents}

\section{Overview} 
The geometric behavior of Hilbert schemes is very much unpredictable and subtle, as Vakil's Murphy's law hold for them (see \cite{Vakil06} for positive dimensional subschemes parametrized by the Hilbert scheme, and \cite{Jelisiejew20} for zero-dimensional ones). As a consequence, there are also several long-standing open conjectures on the geometry of Hilbert schemes.

In this article, we want to understand the singularities of the Hilbert scheme of points via understanding the tangent space. We use the shape of the convex hull of an element in the Hilbert scheme as an optimal singularity detector. Based on this, we suggest conjectural sufficient conditions for having maximal singularity. In particular, in 3D, we (conjecturally) explicitly describe the most singular points of a class of specific given lengths by manipulating the powers of the maximal ideal. Separately, we suggest a necessary condition for having maximum dimension tangent space at a point.

\subsection*{Why do we care about purely conjectural statements?} We emphasize that the conjectures provide a broad extension in a general format of the (counter-)examples provided in \cite{Sturmfels} and \cite{Ramkumar-Sammartano}, where Sturmfels, respectively Ramkumar and Sammartano provided new examples of points of respective lengths $8$ and $39$ in three dimensions with maximum dimension tangent spaces, which contradict relevant conjectures in \cite{Bri-Iar} and \cite{Sturmfels}, respectively. Furthermore, our general format involves convex geometric interpretation of the most singular points using computer algebra-based data (via Macaulay2) and via a mathematical pattern recognition procedure aiming to unify the properties of such points. This new interpretation will help to understand the nature of optimal singularities of the Hilbert schemes better and have a new perspective towards solving long-standing open problems regarding the tangent space and singularities.

\subsection*{Plan}In Section \ref{sec: statement of main conjectures}, first we present conjectural sufficient and necessary conditions in a general format (Section \ref{sec: main}), and in Section \ref{sec: 3Dpartial classification}, we conjecturally provide a partial classification of explicit ideals in 3D satisfying the sufficient conditions conjecture.

In Section \ref{sec: examples}, we provide several explicit examples as evidence for the conjectures in dimension three. In particular, we start with a table in dimension $3$ to present how often our conjectural types  occur for up to length $40$.

\subsection*{Previous work}In \cite[Section 2.6]{Jelisiejew23Open}, some open problems related to 
 the tangent space to the Hilbert schemes of points are discussed. Here, we list some of the relevant references: In \cite{Bri-Iar,Sturmfels,Ramkumar-Sammartano} the maximum dimension tangent space is considered which will be briefly discussed later. Also, see \cite{MNOP1,PandharipandeSlides,Ramkumar-Sammartano1,GGGL23,RicolfiSign, PT, JKSCounterBehrend} for problems motivated by enumerative geometry, including counter-examples for the parity conjecture  for the tangent space in \cite{GGGL23} and for the constancy of the Behrend function in \cite{JKSCounterBehrend}. Some other singularity related problems can be found in \cite{HuX,Jelisiejew20, KatzS, StevensJ, Ranestad-Schreyer}.

\subsection*{Notation} We summarize the notations which will be defined in the next section:
\begin{center}
   \begin{tabular}{ r l }
     $\conv(P)$ &  The convex hull of a compact set $P \subset \BR^N$.\\
      $\conv(I)$ &  The convex hull of \text{$P=$ the set of all of the exponents of}\\&\text{the monomials of a monomial ideal $I$ in $N$ variables}.\\
      $\overline{\partial} \conv(I)$& The upper boundary of the convex hull.\\
    
        $\underline{\partial} \conv(I)$& The lower boundary of the convex hull.\\
      {$\colength(I)$}& {$= \mathrm{dim}_{\BC}\frac{\BC[x_1,x_2,\ldots,x_N]}{I}$.}\\

      \end{tabular}
      \end{center}
      \begin{center}
         \begin{tabular}{ r l }
           $T(I)$& $= \mathrm{hom}(I,\frac{\BC[x_1,x_2,\ldots,x_N]}{I})$, the dimension of the  tangent space to the\\& the Hilbert scheme at the ideal $I$.\\

         $\mathfrak{m}$& $=(x_1,x_2,\ldots,x_N)$, the maximal ideal of  $\BC[x_1,x_2,\ldots,x_N]$.

                      \end{tabular}
     \end{center}
   
   \subsection*{Terminology} By a \textit{mixed monomial generator} for a monomial ideal $I$ in $\BC[x_1,\ldots,x_N]$, we mean a monomial term containing the powers of at least two different $x_i$.  Also, when we say a \text{zero dimensional ideal} $I$ in $R$, it means that $\frac{R}{I}$ is a zero dimensional algebra. {By \textit{a point (or an ideal) having the maximum dimension tangent space}, we mean that the Hilbert scheme has the maximum dimension tangent space at that point (or the subscheme defined by the ideal).}

\subsection*{Acknowledgements} 
I would like to thank Joachim Jelisiejew for encouraging me to write down and publish the conjectures, comments on preliminary versions, and several enlightening discussions. I would like to thank Tony Iarrobino for several helpful comments and suggestions on various versions. I benefited from  useful comments and helpful discussions with Cristina Bertone, David Eisenbud, Rahul Pandharipande, 
Andrea Ricolfi, Alessio Sammartano, Hal Schenck, Ravi Vakil, and Volkmar Welker. This work was partially supported by the NSF Grant No. DMS-1928930, while the author was in residence at  SLMath (formerly known as MSRI) in Berkeley, California, during the Fall
2022 semester. The author was supported by UKRI grant No. EP/X032779/1, and was resident at ETH Z\"urich, during the final stages of writing. {The author was supported by the ERC Advanced Grant MSAG, during revisions.} The general rough ideas were initiated when the author was partially supported by EPSRC grants EP/T015896/1 and EP/W522648/1.
We acknowledge using GeoGebra \cite{geogebra5} and Macaulay2 \cite{GS}.   
\section{Statements of the conjectures} \label{sec: statement of main conjectures}We start with stating the general conjectures in Section \ref{sec: main}, followed by the statement of the conjecture by Briançon-Iarrobino in Section \ref{sec: application}. Then in Section \ref{sec: tangentFat}, we prove a formula for the dimension of the tangent space of specific fat point including the powers of the maximal ideal, as the conjectural formulas for the dimension will be given in terms of these dimensions. In Section \ref{sec: 3Dpartial classification}, we restrict our attention to $N=3$, and explicitly describe the ideals of specific colengths with maximum dimension tangent space.
\subsection{Main Conjectures} \label{sec: main} In this subsection we give the statements of our main conjectures, namely conjecture \ref{combinatorialCrit} and Conjecture \ref{necessaryCondition}. 
\subsection*{Conventions} {Through the paper}, we consider ideals with minimal generators. {We order the variable as expected: $x_1 <\ldots <x_N$}.

\begin{definition}\cite[Theorem 3.15]{Gärtner-Hoffmann}
 For a compact set $P \subset \BR^N$, we can characterize \textit{the convex hull of $P$, $\conv(P)$}, as the smallest (with respect to set inclusion) convex subset of $\BR^N$ that contains $P$.
\end{definition}

\begin{notation}
For a monomial ideal $I$ in $N$ variables, if $P$ is the set of all of the exponents of the monomials of $I$, then we denote the convex hull of $P$, by $\conv(I)$. 
\end{notation}

\begin{definition} \label{Definition:upperLowerBoundary} 
For a compact set $P \subset (\BR^{\geq0})^N$, we define the \textit{lower boundary of $\conv(P)$}, (denoted by $\underline{\partial} \conv(P)$) to be the facets of the convex hull which are visible from the origin of $\BR^N$, $(\underbrace{0,0,\ldots, 0}_{\text{$N$ times}})$. Similarly, we define the \textit{upper boundary of $\conv(P)$}, (denoted by $\overline{\partial} \conv(P)$), to be the facets of the convex hull which are visible from $(\underbrace{+\infty,+\infty,\ldots, +\infty}_{\text{$N$ times}})$. \{By a facet $F$ being \textit{visible from a point $p$}, we mean that there exists an inner point $q$ of
the convex hull such that the segment $\overline{pq}$ intersects the boundary of the convex hull at an inner point of $F$. By monomials being lied \textit{below} a $(N-1)$-simplex, we mean that the points corresponding to those monomials are visible from the origin (See Figure \ref{Fig: definitions}). By being \textit{symmetric} with respect to an axis, we mean if we fix the variable of the corresponding axis (and all the other symmetric axes), and replace any two variables which belong to non-symmetric axes, the equation of the monomial ideal remains unchanged (See Figure \ref{Fig: definitions}). A \textit{missing lattice point} is any possible point with integer coordinates on the convex hull that does not appear as one of the monomial terms of the ideal.
\end{definition}

\begin{remark*}
   Let $F$ be a facet of $\mathrm{Conv}(P)$ which is not contained in any coordinate (hyper)-plane. Then $F$ being visible from
$(+\infty,+\infty,\ldots,+\infty)$ is equivalent to $F$ not being visible from $(0,0,\ldots,0)$, and vice versa.
\end{remark*}

\begin{definition*}[Borel-fixed ideal]\cite{Peeva}
 Let $I$ be a 0-dimensional monomial ideal in $R=\BC[x_1,\ldots,x_N]$. Then we say that $I$ is \textit{Borel-fixed}, if 
 I is a monomial ideal, and if $a$ is a generator of $I$ such that $x_i|a$, then  $x_j(\frac{a}{x_i})\in I$, for all $1\leq j<i$.
\end{definition*}

\begin{remark}
    Let $P$ be the set of all the exponents of the monomials of a monomial ideal $I$ in $\BC[x_1,\ldots,x_N]$. Then $P \subset (\BR^{\geq0})^N$, and so we can apply Definition \ref{Definition:upperLowerBoundary} in this case, which is the case for the rest of this article. In such a case, we denote the lower and the upper boundary of the convex hull by $\underline{\partial} \conv(I)$ and $\overline{\partial} \conv(I)$, respectively. 
\end{remark}

\begin{figure}[H]
 \subcaptionbox*{}[.28\linewidth]{
    \includegraphics[width=\linewidth]{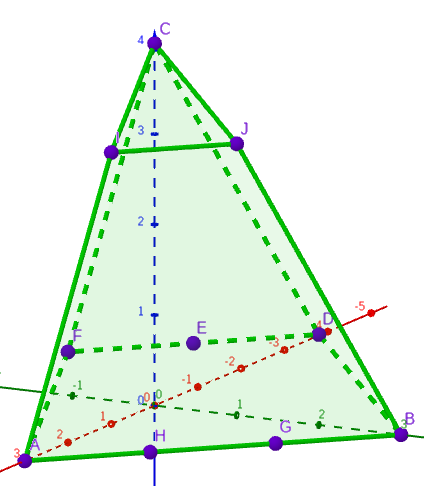}
    
  }
  \hskip15ex
  \subcaptionbox*{}[.21\linewidth]{
    \includegraphics[width=\linewidth]{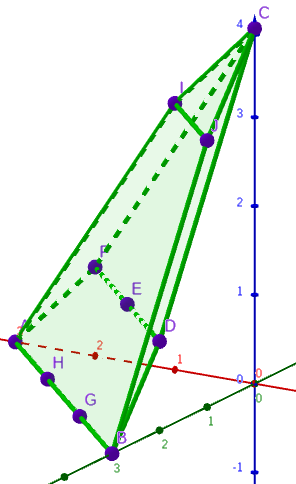}
  }
  \caption{This picture shows two different angles of the convex hull of an ideal. The facets $\mathrm{CDF}$ and $\mathrm{ABDF} $ are \textit{visible from the origin}. The monomials $D$, $E$ and $F$ are \textit{below} the 2-simplex spanned by $A$, $B$ and $C$. Both the lower and the upper boundary of the convex hull are \textit{symmetric} with respect to Axis $z$ (the vertical axis). In case any of the points $D,E,F,G,H$ were missed out, they would have been \textit{missing lattice points}.}
  \label{Fig: definitions}
\end{figure}

Although the claim of the following conjecture will be for the case of $N=3$, we begin with arbitrary $N$ format, and then in Question \ref{question}, we ask if the conjecture holds for arbitrary $N$:

\begin{conj} \label{combinatorialCrit}

Let $N\geq 3$. Suppose that $I$ is a 0-dimensional Borel-fixed ideal in $\BC[x_1,x_2,\ldots,x_N]$, which  admits a convex hull, $\conv(I)$, spanned by all of its monomial generators. There are positive integers $m_1\leq m_2\leq\ldots\leq m_N$ such that $$I=(x_1^{m_1},x_2^{m_2}, \ldots,x_N^{m_N},\text{all the mixed monomial generators}).$$ Suppose that
\begin{align}\label{necessary}
   \text{if ${N-1+k \choose N}\leq\mathrm{colength}(I)<{N+k\choose N}$, then $m_1=k$},
\end{align}
and either of the following three cases occurs:
\begin{itemize}
    \item[I]\label{I}

\begin{enumerate}
\item[($a$)] $m_i$'s satisfy the following conditions

\begin{itemize}
    \item[(i)] at least $N-1$ of the $m_i$'s are equal,
    \item[(ii)] if $m_1=m_2=\ldots =m_{N-1}$, then $ m_N \leq m_{N-1}+2$,
    \item[(iii)] if $m_2=m_3=\ldots=m_N$, then $m_1 \geq m_2-1$,
\end{itemize}
\item[($b$)] $\overline{\partial}\conv(I)$ is the $(N-1)$-simplex spanned by  $m_1, m_2, \ldots, m_N$,
    \item[($c$)]  $\underline{\partial}\conv(I)$ has the maximal number of lattice points among the 0-dimensional minimally generated monomial ideals with the following property: they are generated by $m_1,\ldots,m_N$ plus some points below (or lying on, in case the convex hull is $(N-1)$ dimensional) the $(N-1)$-simplex spanned by $m_1, m_2, \ldots, m_N$,
    \item[($d$)] $\underline{\partial}\conv(I)$ is symmetric with respect to at least one of the coordinate axes,
    \item[($e$)] All the minimal monomial generators in $I$ are contained in  
  $\underline{\partial}\conv(I)$, and there is no missing lattice point on the lower boundary.
\end{enumerate}

\item[II]
\begin{enumerate}
\item[($a^\prime$)] $m_i$'s satisfy the following condition
\begin{itemize}

    \item[(i)]  $m_2=\ldots =m_{N-1}=m_{N}=k+1$, and $ m_1=k$,

\end{itemize}
\item[($b^\prime$)] $\underline{\partial}\conv(I)$ is the $(N-1)$-simplex spanned by  $m_1, m_2, \ldots, m_N$,
    \item[($c^\prime$)] $\overline{\partial}\conv(I)$ has the maximal number of lattice points among the 0-dimensional minimally generated monomial ideals with the following property: they are generated by $m_1,\ldots,m_N$ plus 
    some points above (or on) the $(N-1)$-simplex spanned by $m_1, m_2, \ldots, m_N$,
    \item[($d^\prime$)] $\overline{\partial}\conv(I)$ is symmetric with respect to at least one of the coordinate axes,
    \item[($e^\prime$)] All the minimal monomial generators in $I$ are contained in 
  $\overline{\partial}\conv(I)$, and there is no missing lattice point on the upper boundary.
\end{enumerate}

\item[III]
\begin{enumerate}
\item[($a^{\prime\prime}$)] $m_i$'s satisfy the following conditions

\begin{itemize}
    \item[(i)] either $m_1=m_2=\ldots =m_{N-1}=k$, and $ m_N=k+1$, or
    \item[(ii)]  $m_2=\ldots =m_{N-1}=m_{N}=k+1$, and $ m_1=k$,
\end{itemize}
\item[($b^{\prime\prime}$)] $\overline{\partial}\conv(I)$ is strictly above, and $\underline{\partial}\conv(I)$ is strictly below the $(N-1)$-simplex spanned by  $m_1, m_2, \ldots, m_N$ (the convex hull may intersect the $(N-1)$-simplex spanned by  $m_1, m_2, \ldots, m_N$ only at faces),
    \item[($c^{\prime\prime}$)] The lower and upper boundary of $\conv(I)$ together have the maximal number of lattice points among the 0-dimensional minimally generated monomial ideals with the following property: they are generated by $m_1,\ldots,m_N$ plus
    some points both above and below the $(N-1)$-simplex spanned by $m_1, m_2, \ldots, m_N$,
    \item[($d^{\prime\prime}$)] The convex hull $\conv(I)$ is symmetric with respect to at least one of the coordinate axes,
    \item[($e^{\prime\prime}$)] All the minimal monomial generators in $I$ are contained either on the lower or upper boundary of 
  $\conv(I)$ (and not inside the convex hull), and there is no missing lattice point in the upper or lower boundary.
\end{enumerate}
\end{itemize}

  Then, for $N=3$, the monomial ideal $I$  has the maximum dimension tangent space among all the elements in the Hilbert scheme $\Hilb^{\colength(I)}(\BA^N)$.
\end{conj}

\begin{question} \label{question}
    Does Conjecture \ref{combinatorialCrit} hold for arbitrary $N\geq 3$?
\end{question}
\begin{remark}
Obviously, the conjecture can be stated with respect to any reordering of the $m_i$'s as well.
\end{remark}

The following conjecture states that condition \ref{necessary} in Conjecture \ref{combinatorialCrit} is a necessary condition for a Borel-fixed ideal to have the maximum dimension tangent space.
\begin{conj}[Necessary condition]  \label{necessaryCondition}Let $N\geq 3$.  Suppose that $I$ is a 0-dimensional Borel-fixed ideal in $\BC[x_1,x_2,\ldots,x_N]$ which is  given by
   \begin{align*}
       I=(x_1^{m_1},x_2^{m_2}, \ldots,x_N^{m_N},\text{mixed monomial generators}) 
   \end{align*}
   where  $m_1\leq m_2\leq\ldots\leq m_N$. 

   Then, if ${N-1+k \choose N}\leq\colength(I)<{N+k\choose N}$, and $T(I)$ is maximal, then  $m_1=k$.
  
\end{conj}

The following simple lemma shows that   for  the case of boundary colength, the necessary condition in Conjecture \ref{necessaryCondition} is sufficient as well.

\begin{lemma} \label{necessaryConditionMaximal}
Let $N\geq 3$.  Suppose that $I$ is a 0-dimensional Borel-fixed ideal in $\BC[x_1,x_2,\ldots,x_N]$ which is given by
   $$I=(x_1^{m_1},x_2^{m_2}, \ldots,x_N^{m_N},\text{all the mixed monomial generators}),$$
   where  $m_1\leq m_2\leq\ldots\leq m_N$. If $\colength(I)={N-1+k \choose N}$, then $I=\mathfrak{m}^k$ is the only ideal of this colength for which $m_1=k$.
\end{lemma}
\begin{proof}
   As we assume that $I$ is Borel fixed and $m_1=k$, by \cite[Theorem 15.23 (b)]{EisenbudCA}, we can deduce that all the generating monomials of $I$ are of degree at least $k$. 
   
   On the other hand, we assume $\colength(I)={N-1+k \choose N}$, which is the number of all the possible monomials in $N$ variables of degree less than $k$. As all the generating monomials of $I$ have degree $k$, this implies that $I$ consists of all the possible monomials of degree $k$ (otherwise the colength has to be greater than ${N-1+k \choose N}$), i.e., $I=\mathfrak{m}^k$. 
\end{proof}

\subsection{An application}\label{sec: application} If either Conjecture \ref{combinatorialCrit} or Conjecture \ref{necessaryCondition} holds, then in particular, the following well-known and long-standing conjecture will hold (recall that the sufficiency of Conjecture \ref{necessaryCondition} comes from Lemma \ref{necessaryConditionMaximal}) :
\begin{conj}\cite[Briançon-Iarrobino, 1978]{Bri-Iar} \label{1978Conj} The ideal
 $\mathfrak{m}^k=(x_1,x_2,\ldots,x_N)^k$ has the maximum dimension tangent space among all elements in $\Hilb^{{N-1+k \choose N}}(\BA^N)$.
\end{conj}

\subsection{Tangent space for  $(F_1,\ldots, F_{N})^k$, with $F_i$  homogeneous} \label{sec: tangentFat} We need to compute $T(\mathfrak{m}^k)$, as in the next subsection, we express the conjectural maximum dimension tangent spaces in terms of $T(\mathfrak{m}^k)$ for $N=3$. In this subsection, we compute a more general version in Lemma \ref{lem: tangentspacefatpoints}, and the special case will be presented in Corollary \ref{cor: 3Dtangentspacefatpoints}.
\begin{lemma}
\label{lem: tangentspacefatpoints}
Fix $d\geq 1$, and set $I=(F_1,\ldots,F_{N})^k$ to be a 
0-dimensional 
ideal in $\BC[x_1,\ldots,x_N]$, where $F_i$ is a homogeneous polynomial of degree $g_i$ for $i=1,\ldots,N$. Then

\begin{align*}
    \colength(I)
    &=\left(\prod_{1\leq i\leq N}g_i\right) {k+N-1 \choose N},\\
    T(I)
    &=\left(\prod_{1\leq i\leq N}g_i\right)  {k+N-2 \choose N-1}{k+N-1 \choose N-1}\\
    &=N(\colength(I))\\&+ N{\left(\frac{N+1}{(k-1)(N-1)}\left({k+N-2 \choose N-2}-(N-1)\right)\right)}\left(\prod_{1\leq i\leq N}g_i\right)\left(\sum_{j=1}^{k-1}{j+N-1 \choose N}\right).
\end{align*}
\end{lemma}

\begin{proof}

For $T(I)$, we argue as follows. Consider another polynomial ring $S=\BC[y_1,\ldots,y_N]$ and the morphism $\iota\colon S \to R=\BC[x_1,\ldots,x_N]$ sending $y_i\mapsto F_i$. Since $F_1,\ldots,F_N$ is a complete intersection, using the divisorial criterion for flatness (in the graded case), we see that the ring $R$ is a flat $S$-module. It is finitely generated and graded, so it is a free $S$-module. Then it has rank $g_1\cdots g_N$. This implies two things. First of all, the quotient
\[
\frac{R}{(F_1,\ldots,F_N)^k} = \frac{S}{(y_1,\ldots,y_N)^k} \otimes_S R
\]
is a free module of rank $g_1\cdots g_N$ over $S/(y_1,\ldots,y_N)^k$. Next, since $\iota$ is flat, we have $(F_1,\ldots,F_N)^k = R \otimes_S (y_1,\ldots,y_N)^k$ so that finally
\begin{align*}
\Hom_R\left((F_1,\ldots,F_N)^k, \frac{R}{(F_1,\ldots,F_N)^k}\right) &= \Hom_R\left(R\otimes_S (y_1,\ldots,y_N)^k, \frac{R}{(F_1,\ldots,F_N)^k}\right) \\ &= \Hom_S\left((y_1,\ldots,y_N)^k, \frac{R}{(F_1,\ldots,F_N)^k}\right) \\
&= \Hom_S\left((y_1,\ldots,y_N)^k, \left(\frac{S}{(y_1,\ldots,y_N)^k}\right)^{\oplus g_1\cdots g_N}\right) \\ 
&= \Hom_S\left((y_1,\ldots,y_N)^k, \frac{S}{(y_1,\ldots,y_N)^k}\right)^{\oplus g_1\cdots g_N}.
\end{align*}
    In this way the computation is reduced to computing the tangent space to $(y_1,\ldots,y_N)^k$. Here, the tangent space is graded. It is easy to see that its positive part is zero and that its degree $-1$ part is given by sending any minimal generator of $(y_1,\ldots,y_N)^k$ to any form of degree $k-1$. The degree $0$ part also vanishes, since monomial ideals are isolated torus fixed points. The degree $\leq -2$ part vanishes as well. Thus 
\[
\dim_{\BC} \Hom_S\left((y_1,\ldots,y_N)^k, \frac{S}{(y_1,\ldots,y_N)^k}\right) = {k+N-2 \choose N-1}{k+N-1 \choose N-1},
\]
as required.

As for $\colength(I)$, with a similar argument implies the result (we consider  $\Hom_R\left(R, \frac{R}{(F_1,\ldots,F_N)^k}\right)$ instead of $\Hom_R\left((F_1,\ldots,F_N)^k, \frac{R}{(F_1,\ldots,F_N)^k}\right)$ in the previous argument). Then the second identification for $T(I)$ comes from replacing the colength and simplifying the statement.
\end{proof}

\begin{corollary}
\label{cor: 3Dtangentspacefatpoints}
For $\mathfrak{m}^k$
 in $\BC[x,y,z]$, we have

\begin{align*}
    \colength(\mathfrak{m}^k)
    &= {k+2 \choose 3}, \\
    T(\mathfrak{m}^k)
    &={k+1 \choose 2}{k+2 \choose 2}=3{k+2 \choose 3}+ 6\sum_{j=1}^{k-1}{j+2 \choose 3}.
\end{align*}

\end{corollary} 
\subsection{Conjectural partial classification of ideals of types in Conjecture \ref{combinatorialCrit}, and closed formulas for the dimension of their tangent space in three dimensions}\label{sec: 3Dpartial classification}

When $N=3$, we have the following conjecture:

\begin{conj} \label{dim3}
    Let $k\geq 1$ be an integer. Then for any colength $n={k+2 \choose 3}+i$, for $i=0,1,2,3,4,k+1,2k+1,{k+3 \choose 3}-{k+2 \choose 3}-1$, provided that $n$ is strictly less than ${k+3 \choose 3}$, there exists at least one example of an ideal satisfying conditions in Conjecture \ref{combinatorialCrit} with the maximum dimension tangent space among all the elements in $\Hilb^n(\BA^3)$.
\end{conj}
Note that  $i=0$ is Conjecture \ref{1978Conj}. For other choices of $i$, we have the following conjecture. Note that from Corollary \ref{cor: 3Dtangentspacefatpoints}, we can replace $T(\mathfrak{m}^k)-3\colength(\mathfrak{m}^k)$ everywhere by $6\sum_{j=1}^{k-1}{j+2 \choose 3}$:
\begin{conj} \label{dim3Types}
    Let $n={k+2 \choose 3}+i$. For $i=1,2,3,4,k+1,2k+1,{k+3 \choose 3}-{k+2 \choose 3}-1$, we classify examples of colength $n$ as follows: \begin{itemize}
    \item[(1)] \label{1} Let $$I=(x^k,y^k,z^{k+1},\text{all the mixed monomial generators of $\mathfrak{m}^k$})$$
be an ideal in $\BC[x,y,z]$.  Then $\colength(I)={k+2 \choose 3}+1$, and for such an ideal all the conditions in Conjecture \ref{combinatorialCrit} (of type I($a$)(ii)) hold, and it has the maximum dimension tangent space. Furthermore $$T(I)=3\colength(I)+(T(\mathfrak{m}^k)-3\colength(\mathfrak{m}^k))=T(\mathfrak{m}^k)+3.$$
    \item[(2)] \label{2} Let $$I=(x^k,y^k,z^{k+2},\text{all the mixed monomial generators of $\mathfrak{m}^k$})$$
be an ideal in $\BC[x,y,z]$. Then $\colength(I)={k+2 \choose 3}+2$, and for such an ideal all the conditions in Conjecture \ref{combinatorialCrit} (of type I($a$)(ii)) hold, and it has the maximum dimension tangent space. Furthermore $$T(I)=3\colength(I)+(T(\mathfrak{m}^k)-3\colength(\mathfrak{m}^k))=T(\mathfrak{m}^k)+6.$$
    \item [(3)] \label{3} Let $k\geq 3$, and \begin{align*}
  I=&(x^k,y^k,z^{k+1}, \text{all the mixed monomial generators of $\mathfrak{m}^k$ with replacing $xz^{k-1}$ and}\\&\text{$yz^{k-1}$ by $xz^{k}$ and $yz^{k}$})      
    \end{align*} be an ideal in $\BC[x,y,z]$. Then $\colength (I)={k+2 \choose 3}+3$, and for such an ideal, all the conditions in Conjecture \ref{combinatorialCrit} (of type III($a^{\prime\prime}$)(i)) hold, and it has the maximum dimension tangent space. Furthermore
\begin{align*}
    T(I)=&3\colength(I)+(T(\mathfrak{m}^k)-3\colength(\mathfrak{m}^k)).
\end{align*}

    \item [(4)] \label{4}  Let \begin{align*}
     I=&(x^k,y^k,z^{k+2}, \text{all the mixed monomial generators of $\mathfrak{m}^k$ with replacing $xz^{k-1}$ and}\\&\text{ $yz^{k-1}$ by $xz^{k}$ and $yz^{k}$})   
    \end{align*} be an ideal in $\BC[x,y,z]$. Then $\colength (I)={k+2 \choose 3}+4$, and for such an ideal, all the conditions in Conjecture \ref{combinatorialCrit} (of type I($a$)(ii)) hold, and it has the maximum dimension tangent space. Furthermore
\begin{align*}
    T(I)=&3\colength(I)+(T(\mathfrak{m}^k)-3\colength(\mathfrak{m}^k))+6.
\end{align*}
    \item [(5)] \label{k+1} Let  \begin{align*}
I=&(x^k,y^{k+1},z^{k+1}, \text{all the mixed monomial generators of $\mathfrak{m}^k$ with replacing }\\&\text{$yz^{k-1},y^2z^{k-2}\ldots ,y^{k-1}z$ by $yz^{k},y^2z^{k-1}, \ldots,y^{k}z$})\end{align*} be an ideal in $\BC[x,y,z]$. Then $\colength (I)={k+2 \choose 3}+(k+1)$, and for such an ideal, all the conditions in Conjecture \ref{combinatorialCrit} (of type I($a$)(iii)) hold, and it has the maximum dimension tangent space. Furthermore
\begin{align*}
    T(I)=&3\colength(I)+(T(\mathfrak{m}^k)-3\colength(\mathfrak{m}^k))+k(k-1).
\end{align*}
    \item  [(6)] \label{2k+1} Let
\begin{align*}
    I=&(x^k,y^{k+1},z^{k+1},  \text{all the mixed monomial generators of degree $k$ containing $x^i$ with $i>1$,}\\&\quad \quad \quad \quad \quad \quad \quad\text{and monomials of degree $k+1$ containing $x^i$ for $0\leq i \leq 1$})
\end{align*}
be an ideal in $\BC[x,y,z]$. Then $\colength(I)={2+k\choose 3}+2k+1$, and for such an ideal, all the conditions in Conjecture \ref{combinatorialCrit} (of type III ($a^{\prime\prime}$)(ii)) hold, and it has the maximum dimension tangent space. Furthermore
\begin{align*}
T(I)=&3\colength(I)+(T(\mathfrak{m}^{k})-3\colength(\mathfrak{m}^{k}))+4{k \choose 2}+6.
\end{align*}
    \item [(7)] \label{-1} Let 
\begin{align*}
    I=&(x^k,y^{k+1},z^{k+1},  \text{all the mixed monomial generators of $\mathfrak{m}^{k+1}$ with removing $x^ky,x^kz$})
\end{align*}
be an ideal in $\BC[x,y,z]$. Then $\colength(I)={3+k\choose 3}-1$, and for such an ideal, all the conditions in Conjecture \ref{combinatorialCrit} (of type II) hold, and it has the maximum dimension tangent space. Furthermore
$$T(I)=3\colength(I)+(T(\mathfrak{m}^{k+1})-3\colength(\mathfrak{m}^{k+1}))-k(k+5).$$
\end{itemize}

\end{conj}
Finally, the following conjectures cover some types of ideals with the maximum dimension tangent space in $\BC[x,y,z]$, which are not covered by Conjecture \ref{combinatorialCrit}:
\begin{conj} \label{*Conj}
    Let $j\geq 1$ be an integer, and
\begin{align*}
I=&(x^{2j+1},y^{2j+1},z^{2j+4}, \text{all the mixed monomial generators of $\mathfrak{m}^{2j+1}$ with replacing $xz^{2j}$}\\&\text{ and $yz^{2j}$ by $xz^{2j+1}$ and $yz^{2j+1}$})    
\end{align*} be an ideal in $\BC[x,y,z]$. Then $\colength (I)={2j+3 \choose 3}+5$, and it has the maximum dimension tangent space. Furthermore
\begin{align*}
    T(I)=&3\colength(I)+(T(\mathfrak{m}^{2j+1})-3\colength(\mathfrak{m}^{2j+1}))+6.
\end{align*}
\end{conj}

\begin{conj} \label{**Conj}
    Let $k\geq 3$ be an integer, and
  \begin{align*}
 I=&(x^k,y^{k+1},z^{k+2}, \text{all the mixed monomial generators of $\mathfrak{m}^{k}$ with replacing $yz^{k-1}$ by $xz^{k}$,}\\&\text{and $\{y^iz^{j}\}_{i+j=k}$ by $\{y^{i'}z^{j'}\}_{i'+j'=k+1}$}) \end{align*} be an ideal in $\BC[x,y,z]$. Then $\colength (I)={k+2 \choose 3}+k+3$, and it has the maximum dimension tangent space. Furthermore
\begin{align*}
    T(I)=&3\colength(I)+(T(\mathfrak{m}^{k})-3\colength(\mathfrak{m}^{k}))+{k+2 \choose 2}+{k-2 \choose 2}\\=&\left({k+1 \choose 2}+1\right)\left({k+2 \choose 2}+1\right)+11.
\end{align*}
\end{conj}

\section{Examples} \label{sec: examples} In Section \ref{ex3D}, we present plenty of supportive examples for sufficient conditions in three dimensions. In Section \ref{sec: sufficient not necessary} and Section \ref{sec: necessary not sufficient}, we give examples to emphasize that in general, conjectural sufficient and necessary conditions cannot play the role of each other.
\subsection{Conjectural sufficient condition in three dimensions} \label{ex3D}
We present some specific examples to see how Conjecture \ref{dim3Types} (or Conjecture \ref{*Conj} or Conjecture \ref{**Conj}) works. We have the following table for up to colength 40. By \textit{there is not such an example} in the table, we mean that there is no example with maximum dimension tangent space, satisfying  Conjectures \ref{combinatorialCrit}, \ref{*Conj}, or \ref{**Conj}. 

\vspace{0.5cm}
\begin{tabular}{ |p{0.54cm}||p{13.58cm}|} 
 \hline

 n&An example in $\Hilb^{n}(\BA^3)$  with the maximum dimension tangent space, for which conditions in Conjecture \ref{combinatorialCrit} (*or Conjecture \ref{*Conj}, **or Conjecture \ref{**Conj}) hold\\
 \hline
 1   & $(x,y,z)$   \\
 2&   $(x,y,z^2)$ \\
 3 &$(x,y^2,z^2,yz)$ \\
 4    &$(x,y,z)^2$ \\
 5&   $(x^2,y^2,z^3,xy,yz,xz)$  \\
 6& $(x^2,y^2,z^4,xy,yz,xz)$  \\
 7& $(x^2,y^3,z^3,xz,yz,yz^2,y^2z)$ \\
 8& $(x^2,y^2,z^4,xy,yz^2,xz^2)$\\
  9& $(x^2,y^3,z^3,yz^2,xz^2,y^2z,xy^2,xyz)$\\

10& $(x,y,z)^3$\\ 11& $(x^3,y^3,z^4,yz^2,xz^2,y^2z,xyz,x^2z,xy^2,x^2y)$\\ 12& $(x^3,y^3,z^5,yz^2,xz^2,y^2z,xyz,x^2z,xy^2,x^2y)$\\ 13&$(x^3,y^3,z^4,y^2z,xyz,x^2z,xy^2,x^2y,xz^3,yz^3)$\\ 
14& $(x^3,y^3,z^5,y^2z,xyz,x^2z,xy^2,x^2y,yz^3,xz^3)$, and \\&$(x^3,y^4,z^4,xz^2,xyz,x^2z,xy^2,x^2y,yz^3,y^2z^2,y^3z)$\\   15*& *$(x^3,y^3,z^6,y^2z,xyz,x^2,xy^2,x^2y,yz^3,xz^3)$\\   16**& **$(x^3,y^4,z^5,xyz,x^2z,xy^2,x^2y,yz^3,xz^3,y^2z^2,y^3z)$\\   
 
   17& $(x^3,y^4,z^4,x^2z,x^2y,yz^3,xz^3,y^2z^2,xyz^2,y^3z,xy^2z,xy^3)$\\ 
   18& there is not such an example\\  
     19& $(x^3,y^4,z^4,yz^3,xz^3,y^2z^2,xyz^2,x^2z^2,y^3z,xy^2z,x^2yz,xy^3,x^2y^2)$\\   20& $(x,y,z)^4$\\
   21& $(x^4,y^4,z^5,yz^3,xz^3,y^2z^2,xyz^2,x^2z^2,y^3z,xy^2z,x^2yz,x^3z,xy^3,x^2y^2,x^3y)$\\
  22& $(x^4,y^4,z^6,yz^3,xz^3,y^2z^2,xyz^2,x^2z^2,y^3z,xy^2z,x^2yz,x^3z,xy^3,x^2y^2,x^3y)$\\
   
23& $(x^4,y^4,z^5,y^2z^2,xyz^2,x^2z^2,y^3z,xy^2z,x^2yz,x^3z,xy^3,x^2y^2,x^3y,yz^4,xz^4)$\\
   
24& $(x^4,y^4,z^6,y^2z^2,xyz^2,x^2z^2,y^3z,xy^2z,x^2yz,x^3z,xy^3,x^2y^2,x^3y,yz^4,xz^4)$\\
   25& $(x^4,y^5,z^5,xz^3,xyz^2,x^2z^2,xy^2z,x^2yz,x^3z,xy^3,x^2y^2,x^3y,yz^4,y^2z^3,y^3z^2,y^4z)$\\
    26&there is not such an example\\ 27**&**$(x^4,y^5,z^6,yz^4,y^2z^3,y^3z^2,y^4z,xz^4,x^2z^2,x^3z,xy^3,y^3x,x^2y^2,xyz^2,xy^2z,x^2yz)$\\
   28&there is not such an example\\29&$(x^4,y^5,z^5,x^2z^2,x^2yz,x^3z,x^2y^2,x^3y,yz^4,xz^4,y^2z^3,xyz^3,y^3z^2,xy^2z^2,y^4z,$\\
 &$xy^3z,xy^4)$\\
   30&there is not such an example\\31&there is not such an example\\
  32&there is not such an example\\
33&there is not such an example\\34&$(x^4,y^5,z^5,yz^4,y^2z^3,y^3z^2,y^4z,x^3yz,x^2yz^2,x^2y^2z,xy^2z^2,xyz^3,xy^3z,x^2z^3,$\\&$xz^4,x^2y^3,xy^4,x^3z^2,x^3y^2)$\\35&$(x,y,z)^5$\\

 \end{tabular}

   \begin{tabular}{ |p{0.54cm}||p{13.58cm}|} 
   
36&$(x^5,y^5,z^6,yz^4,y^2z^3,y^3z^2,y^4z,x^3yz,x^2yz^2,x^2y^2z,xy^2z^2,xyz^3,xy^3z,x^2z^3,$\\&$xz^4,x^2y^3,xy^4,x^3z^2,x^3y^2,x^4y,x^4z)$\\
           37&$(x^5,y^5,z^7,yz^4,y^2z^3,y^3z^2,y^4z,x^3yz,x^2yz^2,x^2y^2z,xy^2z^2,xyz^3,xy^3z,x^2z^3,$\\&$xz^4,x^2y^3,xy^4,x^3z^2,x^3y^2,x^4y,x^4z)$\\38&$(x^5,y^5,z^6,xy^4,x^2y^3,x^3y^2,x^4y,xz^5,yz^5,x^2z^3,y^2z^3,xyz^3,x^3z^2,y^3z^2,xy^2z^2,$\\&$x^2yz^2,x^4z,y^4z,x^2y^2z,xy^3z,x^3yz)$\\ 
  39&$(x^5,y^5,z^7,xz^5,yz^5,x^4z,y^4z,x^4y,xy^4,x^3y^2,x^2y^3,x^3z^2,x^2z^3,y^3z^2,y^2z^3,$\\&$x^3yz,xy^3z,xyz^3,x^2y^2z,x^2yz^2,xy^2z^2)$\\ 40*&*$(x^5,y^5,z^8,y^2z^3,xyz^3,x^2z^3,y^3z^2,xy^2z^2,x^2yz^2,x^3z^2,y^4z,xy^3z,x^2y^2z,x^3yx,$\\&$x^4z,xy^4,x^2y^3,x^3y^2,x^4y,yz^5,xz^5)$\\
   \hline

\end{tabular}

\begin{remark}
    We believe that we are able to fill the gaps in the table, by giving other closed formulas similar to those in Conjecture \ref{*Conj} or Conjecture \ref{**Conj} for the types of ideals that happen there (but this may need further checks of higher degree examples with Macaulay2, which may not be practical for large degrees); however, we leave the table as it is to emphasize that it seems to be impossible to formulate all the possible types of ideals in $3$ variables with the maximum dimension tangent space, since new shapes will show up every time one passes $\mathfrak{m}^k$, for each $k$.
\end{remark}

Now, we visualize some of the examples in three dimensions in the table above, to see how Conjecture \ref{combinatorialCrit} or Conjecture \ref{dim3Types}  (or Conjecture \ref{*Conj}) works. All computations have been done via Macaulay2. 

Examples \ref{S} and \ref{RS} below are the cases when $k=2$ and $k=5$, respectively, in Conjecture \ref{4} part (4):

\begin{example}[of type I($a$)(ii) in Conjecture \ref{combinatorialCrit}]\label{S}
    In \cite[Section 2]{Sturmfels}, Sturmfels showed that the following ideal

    \begin{align*}
        I=(x^2,y^2,z^4,xy,xz^2,yz^2)
    \end{align*}
    has the maximum dimension tangent space as an element in $\Hilb^{{8}}(\BA^3)$ (compare this with $I'$ below with more (minimal) generators which has not the maximum dimension tangent space). One can check that Conjecture \ref{combinatorialCrit} holds in this case.

    \begin{figure}[H]
 \subcaptionbox*{}[.22\linewidth]{
    \includegraphics[width=\linewidth]{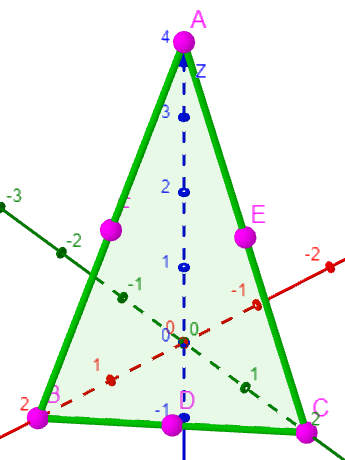}
  }
  \hskip15ex
  \subcaptionbox*{}[.22\linewidth]{
    \includegraphics[width=\linewidth]{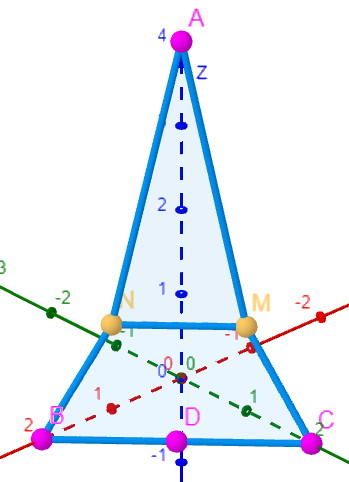}
  }
  \caption{Left: The lower boundary=the upper boundary of the convex hull of $I=(x^2,y^2,z^4,xy,xz^2,yz^2)$. Right: The maximal configuration to check part (c) of Conjecture \ref{combinatorialCrit}(I).}
  \label{Fig: 8}
\end{figure}

      igure \ref{Fig: 8} shows the convex hull of $I$ (left), and the reason why part (c) of Conjecture \ref{combinatorialCrit}(I) holds (right): depicted picture shows the configuration below the 2-simplex $\mathrm{ABC}$ with maximal generators which is equal the number of generators of $I$.

Now compare $I$ with the following example of length $8$ which does not have maximum dimension tangent space: Let $I'=(x^2,y^3,z^3,xy,xz^2,yz^2,y^2z)$. We observe that $\colength(I')=8$, but $T(I)=32$. Note that $I$ is of type III $(a^{\prime\prime})(ii)$, but e.g., $(d^{\prime\prime})$ does not hold (Sturmfels used the ideals $I, I'$ to disprove the conjecture that claimed that an ideal which has the maximum dimension tangent space has to have the maximal number of generators).

\begin{figure}[H]
 \subcaptionbox*{}[.32\linewidth]{
    \includegraphics[width=\linewidth]{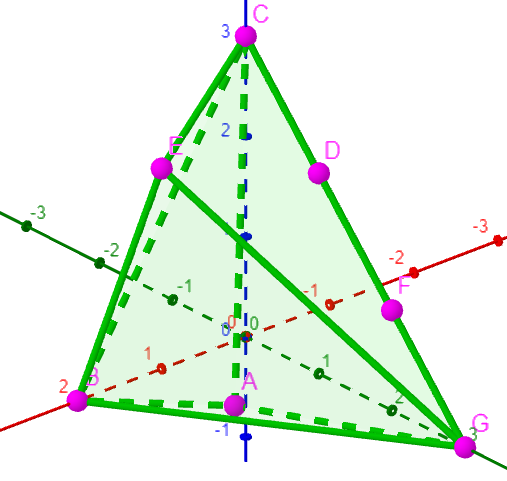}
  }
  
  \caption{$\conv(I')$}
  \label{Fig: 8-n0nMax}
\end{figure}
\end{example}

\begin{example}[of type I($a$)(ii) in Conjecture \ref{combinatorialCrit}]\label{RS}

 Let us consider the following example in $\BC[x,y,z]$, which has been proved to have the maximum dimension tangent space in \cite[Proposition 5.6]{Ramkumar-Sammartano}:
    \begin{align*}
        J=&(x^5,y^5,z^7,xz^5,yz^5,x^4z,y^4z,x^4y,xy^4,x^3y^2,x^2y^3,x^3z^2,\\&x^2z^3,y^3z^2,y^2z^3,x^3yz,xy^3z,xyz^3,x^2y^2z,x^2yz^2,xy^2z^2).
    \end{align*}
    One can show that all the conditions in Conjecture \ref{combinatorialCrit} holds (just need to consider the following picture which shows the lower boundary of the convex hull).
       \begin{figure}[H]
    \centering
    \includegraphics[width=5cm]{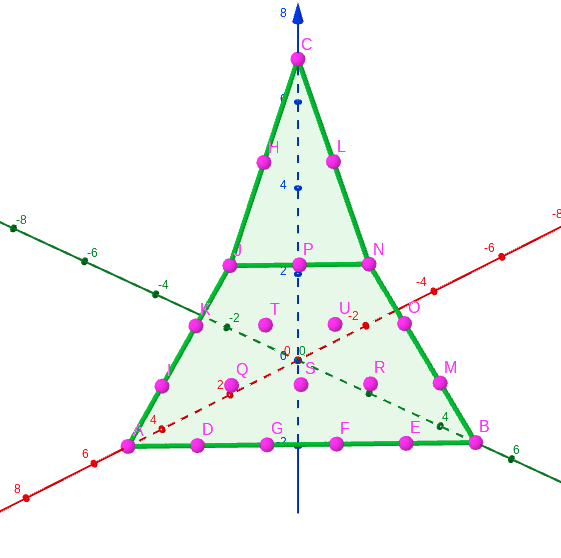}
    \caption{The lower boundary of $\conv(J)$}
\end{figure}
\end{example}

The following example is the case when $k=2$ , in Conjecture \ref{k+1} part (5):

\begin{example} 
Let $L=(x^2,y^3,z^3,xy,xz,yz^2,y^2z)$. Note that if $k=2$ in Conjecture \ref{k+1} part (5), then we have $\colength(L)={k+2 \choose 3}+(k+1)=7$. One can check that the conditions in Conjecture \ref{combinatorialCrit} (of type I($a$)(iii)) satisfy. Also, via Macaulay2, we can check that this ideal has the maximum dimension tangent space, and
\begin{align*}
T(L)=&3\times 7+(18-3\times 4)+2\times1=29.
\end{align*}

The lower boundary of the convex hull is pictured in the following picture:
\begin{figure}[H]
    \centering
    \includegraphics[width=7cm]{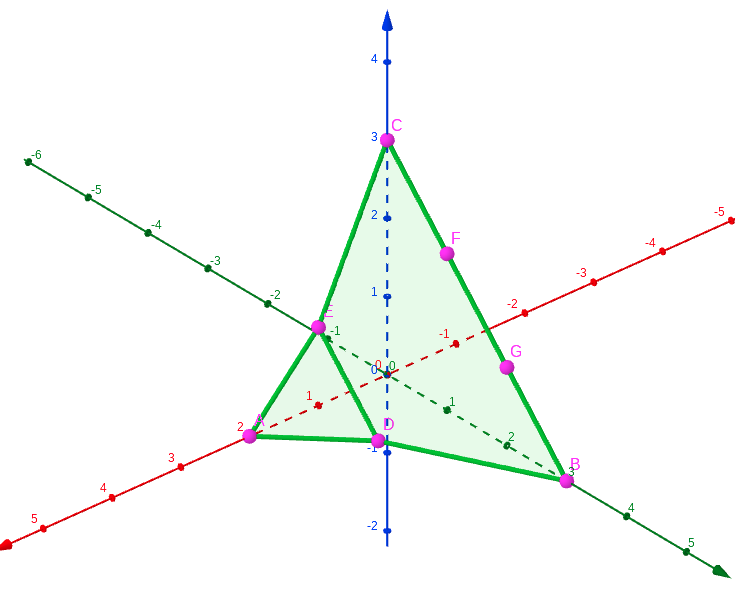}
    \caption{The lower boundary of the convex hull of the ideal $L=(x^2,y^3,z^3,xy,xz,yz^2,y^2z)$}
\end{figure}

\end{example}

The following example Is the case when $k=4$ , in Conjecture \ref{1} part (1):

\begin{example}
Let
$$U=(x^4,y^4,z^5,yz^3,xz^3,y^2z^2,xyz^2,x^2z^2,y^3z,xy^2z,x^2yz,x^3z,xy^3,x^2y^2,x^3y)$$ be an ideal in $\BC[x,y,z]$. We have the following picture for the lower boundary of the convex hull, and one can check that the conditions of Conjecture \ref{combinatorialCrit} (of type I($a$)(ii)) hold. Also, using Macaulay2, we can show that $U$ has the maximum dimension tangent space, and $T(U)=3\colength(U)+(T(\mathfrak{m}^4)-3\colength(\mathfrak{m}^4))=3\times 21+90=153$.
      \begin{figure}[H]
    \centering
    \includegraphics[width=6cm]{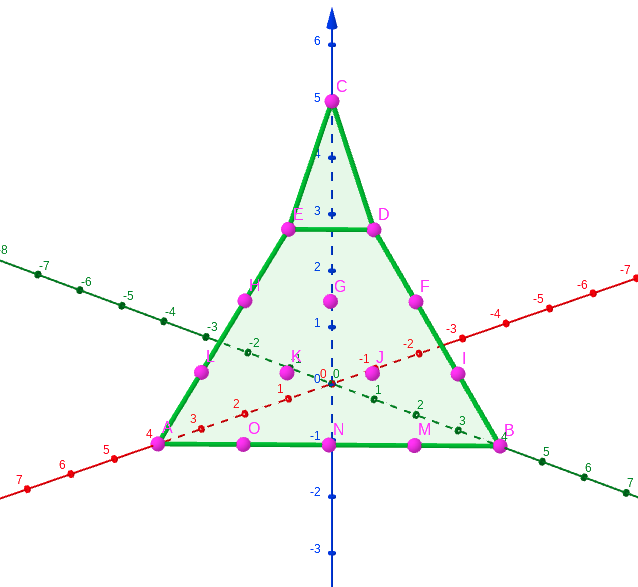}
    \caption{The lower boundary of the convex hull of the ideal $U=(x^4,y^4,z^5,yz^3,xz^3,y^2z^2,xyz^2,x^2z^2,y^3z,xy^2z,x^2yz,x^3z,xy^3,x^2y^2,x^3y)$}
\end{figure}
\end{example}\label{22}
The following example is the case when $k=3$ , in Conjecture \ref{2} part (2):

\begin{example} \label{12Symmetric} Let
\begin{align*}
    F=(x^3,y^3,z^5,yz^2,xz^2,y^2z,xyz,x^2z,xy^2,x^2y)
\end{align*}
     be an ideal in $\BC[x,y,z]$. The lower boundary of the convex hull is shown in the following picture, and one can check that the conditions of Conjecture \ref{combinatorialCrit} (of type I(a)(ii)) hold. Then, by Macaulay2, we can show that $F$ has the maximum dimension tangent space, and $T(F)=3\colength(F)+(T(\mathfrak{m}^3)-3\colength(\mathfrak{m}^3))=3\times 12+30=66$.
       \begin{figure}[H]
    \centering
    \includegraphics[width=6cm]{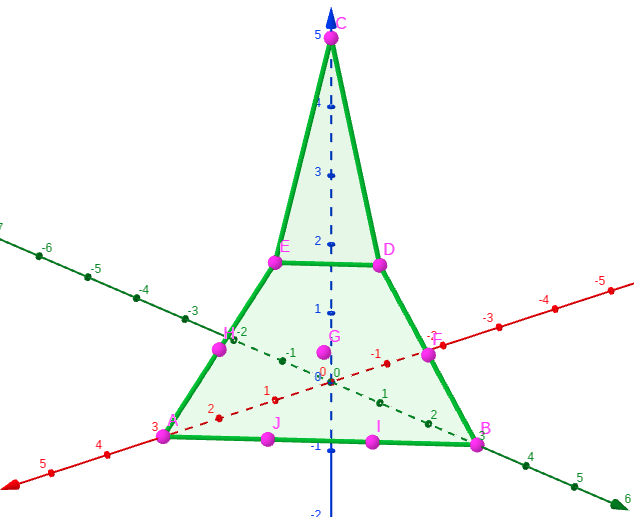}
    \caption{The lower boundary of $\conv(F)$ }
    \end{figure}
\end{example}
The following example Is the case when $k=4$ , in Conjecture \ref{1} part (7):
\begin{example}[of type II in Conjecture \ref{combinatorialCrit}]\label{34} Let
\begin{align*}
    V=&(x^4,y^5,z^5,yz^4,y^2z^3,y^3z^2,y^4z,x^3yz,x^2yz^2,x^2y^2z,xy^2z^2,xyz^3,xy^3z,x^2z^3,xz^4,\\&x^2y^3,xy^4,x^3z^2,x^3y^2)
\end{align*}
be an ideal in $\BC[x,y,z]$. We have $\colength (V)=34$, and it has the maximum dimension tangent space, and $T(V)=276$. The convex hull of $V$ shows that this ideal satisfies type II conditions of Conjecture \ref{combinatorialCrit} (e.g., part $(c^{\prime})$ holds as no 0-dimensional ideal containing the generators $A,B$ and $C$ with the rest of the generators above the triangle $\mathrm{ABC}$ with more than $19$ generators can be made):
     \begin{figure}[H]
    \centering
    \includegraphics[width=6cm]{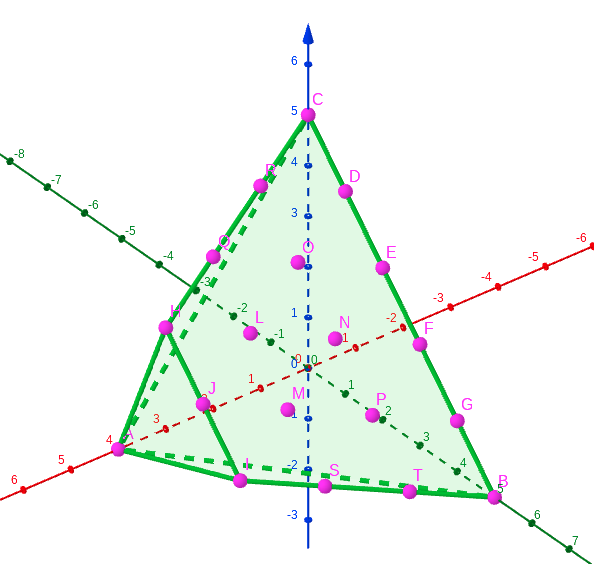}
    \caption{$\conv(V)$}
    \label{V}
    \end{figure}
\end{example}

The following example is the case when $k=5$ , in Conjecture \ref{1} part (3):

\begin{example}[of type III ($a^{\prime\prime}$)(i) in Conjecture \ref{combinatorialCrit}] \label{38} Let
\begin{align*}
    W=&(x^5,y^5,z^6,xy^4,x^2y^3,x^3y^2,x^4y,xz^5,yz^5,x^2z^3,y^2z^3,xyz^3,x^3z^2,y^3z^2,xy^2z^2,x^2yz^2,\\&x^4z,y^4z,x^2y^2z,xy^3z,x^3yz)
\end{align*}
be an ideal in $\BC[x,y,z]$. One can check that $\colength (W)=38$, and it has the maximum dimension tangent space, and $T(W)=324$. The convex hull of $W$ (pictured below) shows that this ideal satisfies type III conditions of Conjecture \ref{combinatorialCrit}:
     \begin{figure}[H]
    \centering
    \includegraphics[width=5cm]{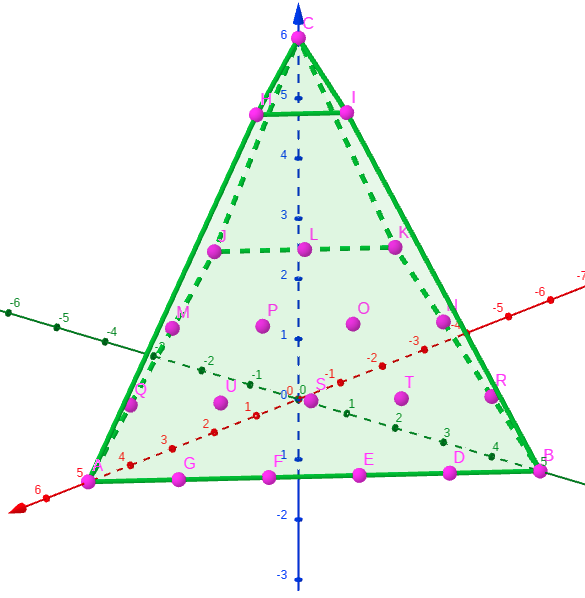}
    \caption{$\conv(W)$}
    \label{W}
    \end{figure}
\end{example}

The following example Is the case when $k=4$ , in Conjecture \ref{1} part (6):

\begin{example}[of type III ($a^{\prime\prime}$)(ii) in Conjecture \ref{combinatorialCrit}]\label{29} Let
\begin{align*}
    O=&(x^4,y^5,z^5,x^2z^2,x^2yz,x^3z,x^2y^2,x^3y,yz^4,xz^4,y^2z^3,xyz^3,y^3z^2,xy^2z^2,y^4z,xy^3z,xy^4)
\end{align*}
be an ideal in $\BC[x,y,z]$. We have $\colength (O)=29$, and using Macaulay2, we can see that it has the maximum dimension tangent space $T(O)=207$. The convex hull below shows that type III conditions of Conjecture \ref{combinatorialCrit} hold for this ideal  (For instance, part $(c^{\prime\prime})$ holds as no 0-dimensional ideal containing the generators $A,B$ and $C$ with other generators lie both above and below the triangle $\mathrm{ABC}$ with more than $17$ generators can be made).
     \begin{figure}[H]
    \centering
    \includegraphics[width=5cm]{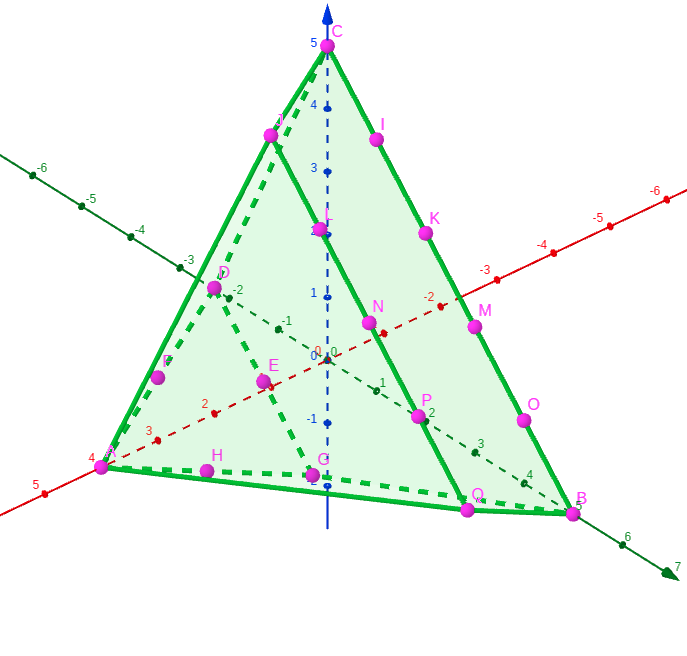}
    \caption{$\conv(O)$}
    \label{O}
    \end{figure}
\end{example}

\begin{example}[Non-example] Let
    \begin{align*}
       Q=(x^4,y^4,z^5,xyz^2,x^2z^2,xy^2z,x^2yz,x^3z,xy^3,x^2y^2,x^3y,yz^4,xz^4,y^2z^3,y^3z^2).
    \end{align*}
    
We have $\colength(Q)=25$ and $T(Q)=165$, and $Q$ does not have maximum dimension tangent space. Note that This ideal is of type III $(a^{\prime\prime})(i)$ in Conjecture \ref{combinatorialCrit}, but parts $(c^{\prime\prime})$ and $(d^{\prime\prime})$ do not hold (see Figure \ref{Fig:25} for a configuration that violates $(c^{\prime\prime})$).
     \begin{figure}[H]
 \subcaptionbox*{}[.28\linewidth]{%
    \includegraphics[width=\linewidth]{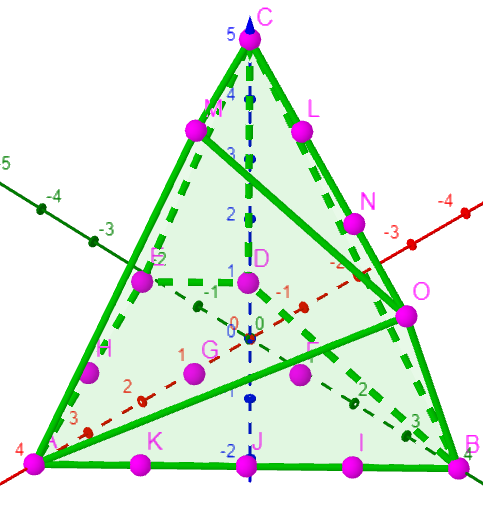}%
  }%
  \hskip15ex
  \subcaptionbox*{}[.28\linewidth]{%
    \includegraphics[width=\linewidth]{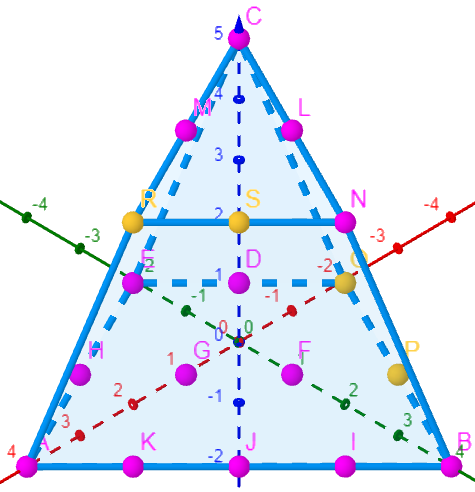}%
  }
  \caption{$\conv(Q)$ with $15$ generators (left) versus a configuration with $18$ generators that contradicts part $(c^{\prime\prime})$ of Conjecture \ref{combinatorialCrit} III (right).}
  \label{Fig:25}
\end{figure}
\end{example}

\begin{example}[of type in Conjecture \ref{*Conj}]Let
\begin{align*}
G=&(x^5,y^5,z^8,y^2z^3,xyz^3,x^2z^3,y^3z^2,xy^2z^2,x^2yz^2,x^3z^2,y^4z,xy^3z,x^2y^2z,x^3yx,\\&x^4z,xy^4,x^2y^3,x^3y^2,x^4y,yz^5,xz^5)\end{align*}
be an ideal in $\BC[x,y,z]$. Then $\colength (G)=40$, and it has the maximum dimension tangent space, and $T(G)=336=3\times 40+(T(\mathfrak{m}^5)-3\colength(\mathfrak{m}^5))+6$. 
    \begin{figure}[H]
    \centering
    \includegraphics[width=5cm]{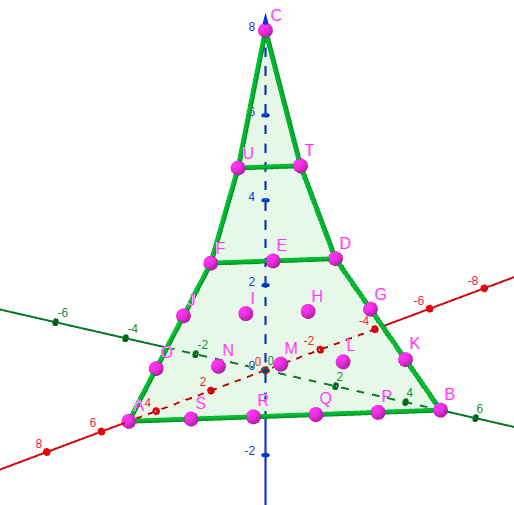}
    \caption{The lower boundary of $\conv(G)$}
    \label{G}
\end{figure}
\end{example}

\begin{example}[of type in Conjecture \ref{**Conj}]\label{**example}
    Let
\begin{align*}
H=&(x^4,y^5,z^6,yz^4,y^2z^3,y^3z^2,y^4z,xz^4,x^2z^2,x^3z,xy^3,y^3x,x^2y^2,xyz^2,xy^2z,x^2yz)\end{align*}
be an ideal in $\BC[x,y,z]$. Then one can check that $\colength (H)=27$, and it has the maximum dimension tangent space, and $T(H)=187=3\times 27+(T(\mathfrak{m}^4)-3\colength(\mathfrak{m}^4))+{4+2\choose 2}+{4-2 \choose 2}$. 
    \begin{figure}[H]
    \centering
    \includegraphics[width=5cm]{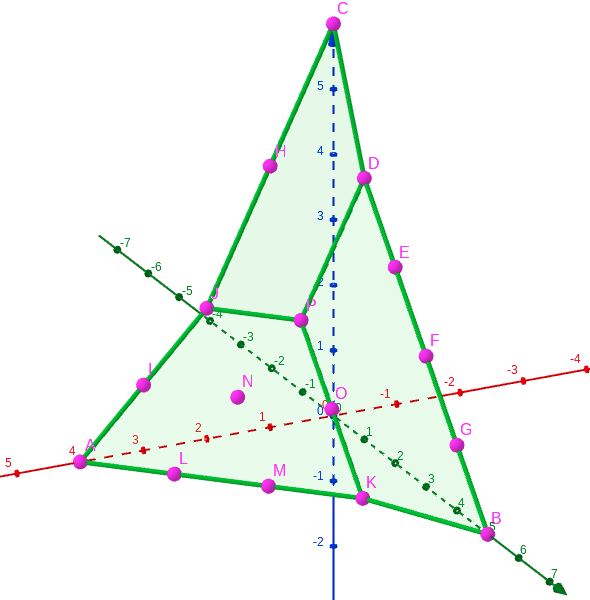}
    \caption{The lower boundary of $\conv(H)$}
    \label{T}
\end{figure}
\end{example}
\subsection{Conjecture \ref{combinatorialCrit} does not give a necessary condition}\label{sec: sufficient not necessary}
 Note that Conjecture \ref{combinatorialCrit} does not provide a necessary condition to have maximum dimension tangent space; apart from examples of types in conjectures \ref{*Conj} and \ref{**Conj} as above and in the table at the beginning of this section, we also present one more example:
    \begin{example}[non-example]
        Let 
\begin{align*}
   M=(x^3,y^3,z^4,xz^2,y^2z,xyz,x^2z,xy^2,x^2y,yz^3)
\end{align*}
be an ideal in $\BC[x,y,z]$. Then one can check by Macaulay2 that $\colength(M)=12$ and it has the maximum dimension tangent space, and $T(M)=66$. However, as it can be seen from the convex hull in the following picture, conditions of Conjecture \ref{combinatorialCrit} do not hold in this case (In Case I, conditions $b,d$ and $e$ are not satisfied. In case II, condition $d^{\prime}$ is not satisfied. In Case III, conditions $d^{\prime\prime}$ is not satisfied).
         \begin{figure}[H]
    \centering
    \includegraphics[width=5cm]{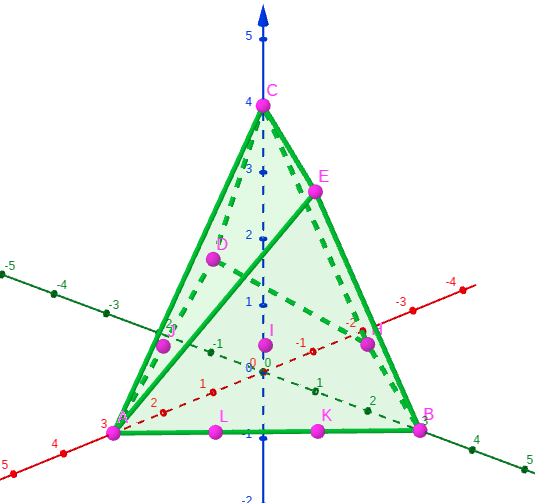}
    \caption{The convex hull of the ideal $M=(x^3,y^3,z^4,xz^2,y^2z,xyz,x^2z,xy^2,x^2y,yz^3)$}
    \label{}
\end{figure}
Note that Example \ref{12Symmetric} presents another example of colength $12$ satisfying conditions of Conjecture \ref{combinatorialCrit}.
    \end{example}

\subsection{Conjecture \ref{necessaryCondition} does not give a sufficient condition in general case} \label{sec: necessary not sufficient} We emphasize that Conjecture \ref{necessaryCondition} does not provide a sufficient condition in general to have maximum dimension tangent space (recall that for certain colengths, Lemma \ref{necessaryConditionMaximal} states that this is a sufficient condition too):

\begin{example}
     Let 
\begin{align*}
   N=(x^3,y^3,z^4,xy^2,x^2y,yz^3,y^2z^2,xyz^2,x^2z)
\end{align*}
be an ideal in $\BC[x,y,z]$. Then one can check by Macaulay2 that $\colength(N)=16$ and it has the tangent space of dimension $T(N)=78$. Although Conjecture \ref{necessaryCondition} holds for this example, $N$ does not have the maximum dimension tangent space (the maximum dimension of the tangent space for colength $16$ is $88$ which is attained by the example for $n=16$ in the table at the beginning of this section).
\end{example}

\appendix
\section{Macaulay2 codes}
Here we add the codes from the already existing package \texttt{StronglyStableIdeals}
 that has been used. For instance, to list the Borel-fixed ideals  of $\colength=16$ in $3$ variables $x,y,z$, and their corresponding dimension of the tangent space, we can use the following.
\vspace{0.2cm}

\begin{verbatim}
loadPackage "StronglyStableIdeals";
L = stronglyStableIdeals(16, QQ[x,y,z,w]);
for I in L list  (I, degree Hom(I, (ring I)/I))
\end{verbatim}

\bibliographystyle{amsplain-nodash}

\bibliography{bib}

\vspace{0.25cm}

\end{document}